\documentclass[12pt,reqno]{amsart}

\usepackage{hyperref}
\usepackage{amsfonts,amstext}
\usepackage{tikz}
\usepackage{amsmath,amstext,amssymb,amscd}
\usepackage{multicol}
\usepackage{calligra}
\usepackage{amsfonts}
\usepackage{amsthm}
\usepackage{amsmath}
\usepackage{amscd}
\usepackage[latin2]{inputenc}
\usepackage{t1enc}
\usepackage[mathscr]{eucal}
\usepackage{indentfirst}
\usepackage{graphicx}
\usepackage{graphics}
\usepackage{pict2e}
\usepackage{epic}
\numberwithin{equation}{section}
\usepackage[margin=2.9cm]{geometry}
\usepackage{epstopdf}

\newcommand{\ints}{\mathbb Z}
\newcommand{\rationals}{\mathbb Q}
\newcommand{\reals}{\mathbb R}
\newcommand{\complexes}{\mathbb C}

\newcommand{\str}{\mathcal{O}}

\newcommand{\proj}{\mathbb{P}}

\theoremstyle{plain}
\numberwithin{equation}{section}
\newtheorem{theorem}{Theorem}[section]
\newtheorem*{theorem*}{Theorem}
\newtheorem{proposition}[theorem]{Proposition}

\newtheorem*{conjecture*}{Nagata Conjecture}
\newtheorem*{conjecture1*}{SHGH Conjecture}
\theoremstyle{definition}
\newtheorem{question}[theorem]{Question}

\newtheorem{remark}[theorem]{Remark}

\DeclareMathAlphabet{\mathcalligra}{T1}{calligra}{m}{n}

\title{Single point Seshadri constants on rational surfaces}
\author[Krishna Hanumanthu]{Krishna Hanumanthu}
\address{Chennai Mathematical Institute, H1 SIPCOT IT Park, Siruseri, Kelambakkam 603103, India}
\email{krishna@cmi.ac.in}
\author[Brian Harbourne]{Brian Harbourne}
\address{Department of Mathematics, University of Nebraska-Lincoln, Lincoln, NE 68588, USA}
\email{brianharbourne@unl.edu}
\thanks{First author was partially supported by a grant from Infosys Foundation}

\date{November 20, 2017}                                           

\begin{document}

\begin{abstract}
Motivated by a similar result of Dumnicki, K\"uronya, Maclean and Szemberg
under a slightly stronger hypothesis, 
we exhibit irrational single-point Seshadri constants on a rational surface $X$
obtained by blowing up very general points of $\proj^2_\complexes$, assuming 
only that all prime divisors on $X$ of negative self-intersection are 
smooth rational curves $C$ with $C^2=-1$. (This assumption is a consequence of
the SHGH Conjecture, but it is weaker than assuming the full conjecture.) 
\end{abstract}
\maketitle

\section{Introduction}
In spite of the many constraints now known on the possible values of Seshadri constants (see for example
\cite{FSST, Han, H, HR}), the longstanding question of whether Seshadri constants on surfaces (defined below) 
can ever be irrational remains open. In the case of a surface $X$ obtained as the blow up $\pi: X \to \proj^2$ 
of the complex projective plane $\proj^2$ at very general points $p_1,\cdots,p_s\in\proj^2$, recent work 
of Dumnicki, K\"uronya, Maclean and Szemberg, \cite[Main Theorem]{DKMS},
shows for $s\geq 9$ that the SHGH Conjecture implies that certain ample divisors $L$ on $X$ 
have irrational Seshadri constants $\varepsilon(X,L,x)$ when $x$ is
a very general point of $X$. In this note we show that less is needed to obtain this conclusion,
namely one merely has to assume that prime divisors $C$ on the blow up $Y$ of $X$ at $x$
with $C^2<0$ satisfy $C^2=C\cdot K_Y=-1$. This assumption is itself a consequence of the SHGH Conjecture
but it is not known to be equivalent to the full SHGH Conjecture, and it leads to a conceptually simpler proof than the
one obtained in \cite{DKMS}. It also leads us to raise the question if an even weaker assumption, viz., Nagata's Conjecture, suffices
to draw the same conclusion.

\section{Main result}

We recall some standard facts.
Given a point $x$ on a smooth projective surface $S$ and an ample divisor $L$,
the Seshadri constant $\varepsilon(S,L,x)$ is defined to be 
$$\varepsilon(S,L,x)=\inf_C \frac{L\cdot C}{{\rm mult}_x(C)},$$
where the infimum is taken over all curves $C$ containing $x$.
Alternatively, let $\pi: Y \to S$ be the blow up
of $S$ at $x$ with exceptional curve $E$. 
Then $\varepsilon=\varepsilon(S,L,x)$ is the supremum of all real $t$ such that
$\pi^*(L) - tE$ is nef and hence $(\pi^*(L) - \varepsilon E)^2\geq0$. 
It follows that $\varepsilon(S,L,x)\leq \sqrt{L^2}$. If $\varepsilon(S,L,x) < \sqrt{L^2}$, 
one says that $\varepsilon(S,L,x)$ is submaximal, in which case
it is well known that there exists a reduced and irreducible
curve $C$ on $S$ passing through $x$ such that $\varepsilon = \varepsilon(S,L,x)  =
\frac{L\cdot C}{{\rm mult}_x(C)}$ (i.e., such that $(\pi^*(L) -
\varepsilon E)\cdot \tilde{C}=0$,
where $\tilde{C}\subset Y$ is the strict transform of $C$).
Such a curve $C$ is called a {\em Seshadri curve} for $L$ at $x$.
Since $\varepsilon=\varepsilon(S,L,x) < \sqrt{L^2}$ implies $(\pi^*(L) - \varepsilon E)^2>0$, it follows 
by the Hodge index theorem that $\tilde{C}^2<0$.

We will also need to refer to multi-point Seshadri constants. 
Given distinct points $p_1,\cdots,p_s$ on $S$ and an ample divisor $L$,
the multi-point Seshadri constant $\varepsilon(S,L,p_1,\cdots,p_s)$ is defined to be 
$$\varepsilon(S,L,p_1,\cdots,p_s)=\inf_C \frac{L\cdot C}{\sum_i{\rm mult}_{p_i}(C)},$$
where the infimum is taken over all curves $C$ containing at least one of the points $p_i$.
Alternatively, let $\pi: Y \to S$ be the blow up
of $S$ at $p_1,\cdots,p_s$ with $E_i$ being the exceptional curve for $p_i$.
Then $\varepsilon=\varepsilon(S,L,p_1,\cdots,p_s)$ is the supremum of all real $t$ such that
$\pi^*(L) - t(E_1+\cdots+E_s)$ is nef and hence $(\pi^*(L) - \varepsilon (E_1+\cdots+E_s))^2\geq0$. 
If $0<t<\varepsilon$, it is easy to see
that $\pi^*(L) - t(E_1+\cdots+E_s)$ is ample (since $F=(t/\varepsilon)(\pi^*(L) - \varepsilon (E_1+\cdots+E_s))$ is nef
and meets any nonnegative linear combination of the $E_i$ positively,
and $\pi^*(L) - t(E_1+\cdots+E_s)=F+(1-(t/\varepsilon))\pi^*(L)$).
When the points $p_i$ are very general, we will write $\varepsilon=\varepsilon(S,L,s)$
for $\varepsilon=\varepsilon(S,L,p_1,\cdots,p_s)$. 

Our focus will be on surfaces $\pi:Y\to X\to\proj^2$ where $X\to\proj^2$ is obtained by blowing up very general points
$p_1,\cdots,p_s$ on $\proj^2$ and $Y\to X$ is the blow up of a very
general point $x\in X$ with exceptional divisor $E$. 
So let $H = \pi^*(\str_{\proj^2}(1))$ and let $E_i$ be the
exceptional curve for each point $p_i$. 
Every divisor on $Y$ is linearly equivalent to a unique integer linear combination
$F=dH-mE-m_1E_1-\cdots-m_sE_s$. 
(Since $Y\to X$ is an isomorphism away from $x$, we can regard
the divisors $H$ and $E_i$ as also being on $X$. With this abuse of notation, 
every divisor on $X$ is linearly equivalent to a unique integer linear combination
$dH-m_1E_1-\cdots-m_sE_s$.)
Such a divisor $F$ is in 
{\it standard} form if $m\geq m_1\geq\cdots\geq m_s\geq 0$ and $d\geq m+m_1+m_2$.
An {\it exceptional} curve on $X$ (or $Y$) is a reduced and 
irreducible rational curve $C$ with $C^2 = -1$ (and hence $-K_X\cdot C=1$, or $-K_Y\cdot C=1$ respectively). 
If $F$ is in standard form, then $F \cdot C \geq 0$ for all exceptional curves $C$ on $Y$. 
(To see this, 
let $F=dH-mE-m_1E_1-\cdots-m_sE_s$ be divisor on $Y$. If $F$ is in
standard form and if $C$ is one of the exceptional curves $E,E_1,\cdots,E_s$ then
clearly $F\cdot C \ge 0$. So suppose that $C$ is different from $E,E_1, \cdots, E_s$. Note
that $F$ is in standard form
if and only if $F$ is a nonnegative linear integer combination of 
$H_0=H$, $H_1=H-E$, $H_2=2H-E-E_1$, $H_3=3H-E-E_1-E_2$, $\cdots$, $H_{s+1}=3H-E-E_1-\cdots-E_s=-K_Y$.
But $H_i$ is nef for $i=0,1,2$
and $H_i\cdot C\geq -K_Y\cdot C=1$ for $i\geq 3$.)

The above definition of standard divisors also extends to divisors with
coefficients in $\rationals$ or $\reals$. If $F$ is a standard
$\rationals$-divisor, then for a suitable positive integer $n$, the
$\ints$-divisor $nF$ is standard. It follows that 
$F\cdot C\ge 0$ for all exceptional curves $C$ on $Y$. If $F$ is a
standard $\reals$-divisor, then $F$ is the limit of a sequence of
standard $\rationals$-divisors. So again 
$F\cdot C\ge 0$ for all exceptional curves $C$ on $Y$.

\begin{proposition}\label{main1}
Let $s \geq 13$ be an integer with $s \ne 15, 16$.
Let $X \to \proj^2_\complexes$ be the blow up of
$\proj^2=\proj^2_\complexes$ at $s$ very general points $p_1,\cdots,p_s$ and let $Y \to X$ be
the blow up of $X$ at a very general point $x \in X$. Suppose that every
reduced and irreducible curve $C$ on $Y$ with $C^2 < 0$ is an
exceptional curve. Then there exists an ample line bundle $L$ on $X$
such that the Seshadri constant
$\varepsilon(X,L,x)$ is irrational for any very general point $x \in X$.
\end{proposition}

\begin{proof}
Let $L = dH - E_1-\cdots-E_s$ be a divisor on $X$ with 
$4d-3 \le s < d^2$. 
By \cite[Corollary]{K} and \cite[Theorem]{X}, $L$
is ample. Let $x$ be a very general point of $X$ and let $\pi: Y \to
X$ be the blow up at $x$ with exceptional curve $E$. 

We will show that there are no Seshadri curves for $L =
dH-E_1-\cdots-E_s$ at $x$ if $4d-3 \le s < d^2$. 
If there were a Seshadri curve $C$, then 
$\varepsilon = \varepsilon(X,L,x) < \sqrt{L^2}=\sqrt{d^2-s}$, 
so $0=(\pi^*(L) - \varepsilon E)\cdot \tilde{C}>(\pi^*(L) - \sqrt{d^2-s}E)\cdot \tilde{C}$.
Since $\tilde{C}^2 < 0$, by hypothesis we have that $\tilde{C}$ is an exceptional curve. 
But note that $\pi^*(L)-\sqrt{d^2-s}E = dH-\sqrt{d^2-s}E-E_1-\cdots - E_s$ is in  standard form: 
since $4d-3 \le s$, we get $(d-2)^2 > d^2-s$, so we have $d > \sqrt{d^2-s}+2$,
and $d^2>s$ so $d^2-s\ge1$, hence $\sqrt{d^2-s}\ge 1$.
It follows that $\pi^*(L)-\sqrt{d^2-s}E$ meets all exceptional curves nonnegatively.
Since $\tilde{C}^2 < 0$, by hypothesis we must have that $\tilde{C}$ is an exceptional
curve. But then $(dH-\sqrt{d^2-s}E-E_1-\cdots - E_s) \cdot
\tilde{C} < 0$ is not possible. Thus $\varepsilon(X,L,x)$ cannot be
submaximal, so 
$\varepsilon(X,L,x)= \sqrt{L^2} = \sqrt{d^2-s}$.

Alternatively, we can directly obtain the equality 
$\varepsilon(X,L,x) =
\sqrt{L^2} = \sqrt{d^2-s}$ when
$4d-3 \le s < d^2$, using the following argument suggested by the
referee. It suffices to show that 
$\pi^{\star}L-\sqrt{d^2-s}E$ is nef. Recall that a line bundle on a
surface is nef if its intersection with every curve of negative
self-intersection is nonnegative. Note that 
$\pi^{\star}L-\sqrt{d^2-s}E$ is in standard form, as shown above. Hence it
intersects all exceptional curves on $Y$ nonnegatively. By assumption
there are no other curves of negative self-intersection on $Y$. 
Thus  $\pi^{\star}L-\sqrt{d^2-s}E$ is nef and hence 
$\varepsilon(X,L,x) =
\sqrt{L^2} = \sqrt{d^2-s}$.

If $s\geq 13$ but $s\neq15,16$, we now show that $d$ can be chosen so that 
$\sqrt{d^2-s}$ is irrational. For $s=13$ or 14, take $d=4$; then $13=4d-3 \le s < d^2=16$,
so $d^2-s=3$ or 2, hence $\sqrt{d^2-s}$ is irrational.
For $s\geq17$, there is always a $d$ with
$4d-3\leq s\leq 6d-10$, since $4d-3=17$ for $d=5$, while $4(d+1)-3\leq(6d-10)+1$ for $d\geq5$.
Thus $(d-3)^2+1=d^2-(6d-10)\leq d^2-s\leq d^2-(4d-3)=(d-2)^2-1$, so $\sqrt{d^2-s}$ again is irrational.
\end{proof}

\begin{proposition}\label{remaining-cases}
Let $X \to \proj^2$ be the blow up of  $\proj^2$ at $s$ very general
points where $s \in \{9,10,11,12,15,16\}$. 
Let $Y \to X$ be the blow up of $X$ at a very general point $x \in X$. 
Suppose that any irreducible and reduced curve on $Y$ of
negative self-intersection is exceptional. Then there is an ample line
bundle $L$ on $X$ such that $\epsilon(X,L,x)$ is irrational. 
\end{proposition}
\begin{proof} We consider different cases.

\smallskip
\underline{$s=9$}:
Let $L = (3n+1)H-n(E_1+\cdots+E_9)$ for $n \geq 1$. Then $L^2 = 6n+1 > 0$. Since
$\varepsilon(\proj^2,\str_{\proj^2}(1),9) = 1/3$, it follows that $L$
is ample. Let $\pi: Y \to X$ be the blow up at a very general point $x \in
X$ with exceptional curve $E$ and let $\varepsilon = \varepsilon(X,L,x)$. 

Note that $\pi^*(L)-\sqrt{6n+1}E=(3n+1)H-n(E_1+\cdots+E_9)-\sqrt{6n+1}E$ is in standard form 
for $n\geq7$ if we take the blow ups in the order $E_1,\cdots,E_9, E$, since $3n+1>n+n+n$ and
$n\geq\sqrt{6n+1}\geq0$.
Now by the same argument used in the
proof of Proposition \ref{main1}, we conclude that
$\pi^*(L)-\sqrt{6n+1}E$ cannot meet any exceptional curve
negatively. Hence   $\epsilon(X,L,x)$ has to be maximal. Thus 
$ \epsilon(X,L,x)$ is irrational provided $L^2=6n+1$ is not a perfect
square for some $n\geq7$. This is the case for example for $n=6m^2$ for any $m\geq 2$.

\smallskip
\underline{$s=10$}:
Let $L = 10H - 3(E_1+\cdots+E_{10})$. Then $L^2 = 10$. 
By hypothesis every curve on $Y$ of negative self-intersection is
exceptional. Clearly the same statement holds on $X$. Under this
hypothesis, it is easy to show that the multi-point Seshadri constant 
$\varepsilon(\proj^2,\str_{\proj^2}(1),10) = 1/\sqrt{10}$. It then follows
that $L$ is ample. 

Note that $\pi^*(L)-\sqrt{10}E$ is in standard form
(since $10 \geq \sqrt{10}+6$). Hence by the same argument used above, we conclude that
$\pi^*(L)-\sqrt{10}E$ cannot meet any exceptional curve
negatively. Thus $\varepsilon(X,L,x) = \sqrt{10}.$ 

\smallskip
\underline{$s=11$}:
Let $L = 7H-2(E_1+\cdots+E_{11})$. The same argument as in the case $s=10$
works to give $\varepsilon(X,L,x) = \sqrt{5}$.

\smallskip
\underline{$s=12$}:
Let $L = 11H-3(E_1+\cdots+E_{12})$. The same argument as in the case $s=10$
works to give $\varepsilon(X,L,x) = \sqrt{13}$.

\smallskip
\underline{$s=15$}:
Let $L = 13H-3(E_1+\cdots+E_{15})$. The same argument as in the case $s=10$
works to give $\varepsilon(X,L,x) = \sqrt{34}$.

\smallskip
\underline{$s=16$}: Let $L = (4n+1)H-n(E_1+\cdots+E_{16})$. Then
a similar argument as in the case $s=9$ shows that  $L$ is ample and
$\varepsilon(X,L,x)$ cannot be submaximal for any $n \geq 9$. So 
$\varepsilon(X,L,x) = \sqrt{L^2} = \sqrt{8n+1}$. This is irrational
for infinitely many $n \geq 9$. 
\end{proof}

\begin{remark}\label{FewPtsRemark}
As is well known to experts \cite{S}, all
single-point Seshadri constants on a blow up of $\proj^2$ at $s \le 8$
points are rational. For $s\leq 7$, this is because the subsemigroup of effective divisor 
classes of an 8 point blow up $S$ of $\proj^2$ is
finitely generated, hence the nef cone is finite polyhedral with boundaries defined
by negative effective classes and effective classes of self-intersection 0. The case of $s=8$ is slightly more delicate
since the subsemigroup of effective divisor 
classes of a 9 point blow up $S$ of $\proj^2$ need not be
finitely generated, but it is generated by the exceptional curves and curves which occur as components
of curves in the linear system $|-K_S|$, so again the nef cone has boundaries defined
by negative effective classes and effective classes of self-intersection 0. 
\end{remark}

Combining Remark \ref{FewPtsRemark}, Proposition \ref{main1} and Proposition \ref{remaining-cases},
we obtain our main theorem. 

\begin{theorem}\label{main}
Let $s \geq 0$ be an integer.
Let $X \to \proj^2$ be the blow up of
$\proj^2$ at $s$ very general points $p_1,\cdots,p_s$ and let $Y \to X$ be
the blow up of $X$ at a very general point $x \in X$. Suppose that every
reduced and irreducible curve $C$ on $Y$ with $C^2 < 0$ is an
exceptional curve. Then there exists an ample line bundle $L$ on $X$
such that the Seshadri constant $\varepsilon(X,L,x)$ is irrational
if and only if $s\geq9$.
\end{theorem}

\begin{remark}
In fact using 
the ideas in the proof of Proposition \ref{main1} and
Proposition \ref{remaining-cases}, we can get the following stronger
assertion.

Let $s\geq 9$
be an integer. Consider the divisor
$L_{d,n} = dH-n(E_1+\cdots+E_s)$ on the blow up $X$ of $\proj^2$ at
$s$ very general points. Let $Y \to X$ be the blow up at a very general
point. Suppose that every reduced and irreducible curve
of negative self-intersection on $Y$ is an exceptional curve. Then for
infinitely many values of $n$, there exists a $d$ such that 
$L_{d,n}$ is ample and
the Seshadri constant $\varepsilon(X,L_{n,d},x)$ is irrational for a
very general point $x \in X$. 
\end{remark}

Our results depend only on assuming all negative curves are exceptional.
A somewhat weaker result was conjectured by Nagata \cite{N}, namely
for a blow up $S$ of $\proj^2$ at $s\geq 10$ very general points,
if $dH-(m_1E_1+\cdots+m_sE_s)$ is linearly equivalent to an effective divisor,
then $d\sqrt{s}\geq \sum_im_i$. 
This is equivalent to conjecturing that $F_0=\sqrt{s}H-E_1-\cdots-E_s$ is nef.
Note for arbitrarily small $\delta>0$ that $F_\delta=(\delta+\sqrt{s})H-E_1-\cdots-E_s$
is rational and semi-effective (meaning that a positive integer multiple is 
linearly equivalent to an effective divisor, which follows since 
$F^2>0$). Thus if $F_0$ is not nef, then there is a prime divisor $C$ with $C^2<0$ and $C\cdot F_0<0$.
From this we see that the SHGH Conjecture implies Nagata's Conjecture.
In fact, if $C$ being a prime divisor with $C^2<0$ implies
$C^2=C\cdot K_S=-1$, then already Nagata's Conjecture is true.
This is because if $C^2<0$ for a prime divisor $C$, then
$C\cdot (\sqrt{s}H-E_1-\cdots-E_s)\geq C\cdot (3H-E_1-\cdots-E_s)\geq 1$.

Thus Nagata's Conjecture is weaker than the assumption we used.
Note further that the Nagata Conjecture exhibits irrational
multi-point Seshadri constants on 
$\proj^2$, since it is equivalent to the statement that $\varepsilon(\proj^2,\str_{\proj^2}(1),s) =
1/\sqrt{s}$ for every $s \ge 10$. 
These remarks raise the following question. 
\begin{question}
Is it possible to exhibit irrational single-point Seshadri constants
on very general blow ups of $\proj^2$  assuming only the Nagata Conjecture? 
\end{question}

{\bf Acknowledgement:} We thank Tomasz Szemberg for reading this
paper and giving useful suggestions. We also thank the referee for
giving an alternate argument in the proof of Proposition \ref{main1}
and numerous other suggestions which improved the exposition.

\end{document}